\documentclass[11pt]{amsart}

\usepackage[T1]{fontenc}
\usepackage[utf8]{inputenc}
\usepackage{lmodern}
\usepackage{amsmath,amssymb,amsthm,mathtools}
\usepackage{microtype}
\usepackage{tikz}
\usetikzlibrary{fadings}
\usepackage{booktabs}
\usepackage[hidelinks]{hyperref}
\usepackage[table]{xcolor}

\allowdisplaybreaks

\theoremstyle{plain}
\newtheorem{theorem}{Theorem}[section]
\newtheorem{lemma}[theorem]{Lemma}
\newtheorem{proposition}[theorem]{Proposition}
\newtheorem{corollary}[theorem]{Corollary}
\newtheorem{conjecture}[theorem]{Conjecture}

\theoremstyle{remark}
\newtheorem{remark}[theorem]{Remark}

\title[Affine Scaling of Jacobi Zeros]
{Affine Scaling of Jacobi Zeros: Sharp Orderings Beyond Gautschi's Conjectures}

\author{K. Castillo}

\address{CMUC, Department of Mathematics, University of Coimbra,
3000-143 Coimbra, Portugal}

\email{math@keniercastillo.com}

\author{F. R. Rafaeli}

\address{Instituto de Matemática e Estatística,
Universidade Federal de Uberlândia (IME-UFU),
Uberlândia, Minas Gerais, Brazil}

\email{rafaeli@ufu.br}

\date{}

\subjclass[2020]{Primary 33C45; Secondary 34C10, 65D32, 26D05}

\keywords{Jacobi polynomials, zeros, Sturm comparison theorem,
Gautschi's conjectures, affine scaling, Liouville normal form,
spherical cubature}

\begin{document}

\begin{abstract}
We settle two conjectures of Gautschi on the degree dependence of the zeros of the Jacobi polynomials $P_n^{(\alpha,\beta)}$, $\alpha,\beta>-1$, and obtain results substantially stronger than those conjectured. The conjectures stem from a line of questions originating in spherical cubature and hyperinterpolation. A Liouville transformation and Sturm comparison yield affine comparison principles with exact thresholds for the pointwise monotonicity of the rescaled potential. We prove that an increasing affine ordering with a degree-independent shift exists if and only if $|\beta|\leq1/2$. For the spectral scale $n+(\alpha+\beta+1)/2$, we determine the exact parameter regions for the two opposite orderings and show that no uniform spectral ordering is possible outside them. We also characterise all equality cases and derive finite-degree bounds in terms of Bessel zeros. The resulting classifications are exact and cannot be enlarged: outside the stated parameter regions the corresponding uniform zero orderings necessarily fail.
\end{abstract}

\maketitle

\section{Introduction}

We write $P_n^{(\alpha,\beta)}$ for the Jacobi polynomial of degree
$n$ in the conventional normalisation
\[
P_n^{(\alpha,\beta)}(1)
=
\frac{(\alpha+1)_n}{n!}.
\]
Throughout the paper we assume that $\alpha,\beta>-1$. In this
Jacobi orthogonality range, the zeros of
$P_n^{(\alpha,\beta)}$ are simple and lie in $(-1,1)$. We write
\[
x_{n,k}^{(\alpha,\beta)}
=\cos\theta_{n,k}^{(\alpha,\beta)},
\quad
0<\theta_{n,1}^{(\alpha,\beta)}
<\cdots<
\theta_{n,n}^{(\alpha,\beta)}
<\pi.
\]
Our purpose is to compare these angles in consecutive degrees after
an affine rescaling of the degree. The two conjectures formulated by
Gautschi in~\cite{Gautschi2009a,Gautschi2009b} provide the principal
motivation for this study. 

These questions originated in polynomial comparisons motivated by
numerical analysis on spheres. Positive-weight cubature formulae are
used both to approximate surface integrals and to construct
hyperinterpolation operators, in which the Fourier--Laplace
coefficients of the orthogonal projection are replaced by their
cubature approximations. Under an additional quadrature-regularity
hypothesis, Le Gia and Sloan proved that, in arbitrary dimension, the
uniform norm of the hyperinterpolation operator has the same optimal
order of growth as that of the corresponding orthogonal
projection~\cite{LeGiaSloan}. The condition is local: for a rule with
$N$ nodes on a $d$-dimensional sphere, it bounds the total quadrature
weight carried by every spherical cap of angular radius $N^{-1/d}$.
Reimer independently obtained the higher-dimensional estimate and
proved that no additional regularity hypothesis is needed for
positive-weight rules admissible for hyperinterpolation; a key
ingredient is a universal cap-weight estimate at inverse
polynomial-degree scale~\cite{Reimer}. Hesse and Sloan later
introduced the related degree-scale condition Property~(R) and used
it in their analysis of worst-case cubature error in Sobolev spaces
on the two-sphere~\cite{HesseSloan2005,HesseSloan2006}.

Leopardi~\cite{Leopardi} subsequently sought explicit estimates for
the constants in Property~(R). For $\xi,\eta\in S^d$, where
$d\geq2$, the reproducing kernel of the spherical polynomial space of
degree at most $n$ is a positive constant multiple of
\[
\frac{P_n^{(d/2,d/2-1)}\bigl(\langle\xi,\eta\rangle\bigr)}
     {P_n^{(d/2,d/2-1)}(1)}.
\]
If $\varphi$ denotes their angular distance, then
$\langle\xi,\eta\rangle=\cos\varphi$. Consequently, his local
cap-weight estimate depends on this normalised Jacobi polynomial and
on its largest zero. The argument also uses Jacobi polynomials of the
adjacent type $P_n^{(1+d/2,d/2)}$, together with a Sturm
comparison~\cite[(2.3), Lemma~3.1, and Theorem~5.1]{Leopardi}.
Leopardi
conjectured that, for $\alpha\geq1/2$, $n\geq1$, and
$0<\theta<\theta_{1,1}^{(\alpha,\alpha-1)}$,
\[
\frac{
P_n^{(\alpha,\alpha-1)}
\left(\cos\frac{\theta}{n}\right)
}{
P_n^{(\alpha,\alpha-1)}(1)
}
<
\frac{
P_{n+1}^{(\alpha,\alpha-1)}
\left(\cos\frac{\theta}{n+1}\right)
}{
P_{n+1}^{(\alpha,\alpha-1)}(1)
}.
\]
For $\alpha=d/2$, this comparison would improve his explicit
Property~(R) constants by the multiplicative factor
$(\operatorname{sinc}(\delta/2))^{d+1}$, where the spherical caps
have angular radius $\delta/m$, with $m$ denoting the degree of
exactness of the cubature rule, and
$\operatorname{sinc}u=\sin u/u$~\cite[Conjecture~5.2 and
Corollary~5.3]{Leopardi}. This proposed improvement depends on the
polynomial-value comparison and does not follow from a comparison of
the zeros alone.

Gautschi and Leopardi studied the preceding polynomial inequality
alongside the corresponding largest-zero inequality
\[
n\theta_{n,1}^{(\alpha,\beta)}
<
(n+1)\theta_{n+1,1}^{(\alpha,\beta)},
\]
and extended both questions to general Jacobi
parameters~\cite{GautschiLeopardi}. For the general-parameter
polynomial inequality, Koumandos proved the comparison for every
$n\geq1$ and $0<\theta<\pi$ at the three Chebyshev parameter
pairs~\cite{Koumandos}:
\[
(\alpha,\beta)\in
\left\{
\left(\frac12,\frac12\right),
\left(\frac12,-\frac12\right),
\left(-\frac12,\frac12\right)
\right\}.
\]
The two problems are logically distinct: one
compares normalised polynomial values, whereas the other compares
first angular zeros. The present paper concerns only the latter.
After this joint work, Gautschi refined the conjectured parameter
domain for the largest-zero inequality and then extended the
comparison to every zero~\cite{Gautschi2008,Gautschi2009a}. On the basis of the computations
reported in~\cite[Section~3]{Gautschi2009a}, he proposed the following
parameter region
\[
\mathcal D_1
\coloneqq
\left\{
(\alpha,\beta)\in(-1,\infty)^2:
|\beta|\leq\frac12,\,\,
\alpha+\beta+1\geq0
\right\}
\setminus
\left\{
\left(-\frac12,-\frac12\right)
\right\}.
\]

\begin{conjecture}[Gautschi, 2009]\label{conj:first}
For every $(\alpha,\beta)\in\mathcal D_1$, $n\geq1$, and
$1\leq k\leq n$,
\[
n\theta_{n,k}^{(\alpha,\beta)}
<
(n+1)\theta_{n+1,k}^{(\alpha,\beta)}.
\]
\end{conjecture}

Prompted by a question of Askey, recorded
in~\cite[Section~1]{Gautschi2009b}, Gautschi next investigated the
shifted factor
\[
\rho\coloneqq\frac{\alpha+\beta+1}{2},
\quad
\gamma_n\coloneqq n+\rho.
\]
For this shifted scale, the conjectured inequality is reversed. Put
\[
\mathcal D_2
\coloneqq
\left\{
(\alpha,\beta)\in(-1,\infty)^2:
|\alpha|\geq\frac12,\,\,
|\beta|\geq\frac12
\right\}
\setminus
\left\{-\frac12,\frac12\right\}^{\!2}.
\]

\begin{conjecture}[Gautschi, 2009]\label{conj:second}
For every $(\alpha,\beta)\in\mathcal D_2$, $n\geq1$, and
$1\leq k\leq n$,
\[
\gamma_{n+1}\theta_{n+1,k}^{(\alpha,\beta)}
<
\gamma_n\theta_{n,k}^{(\alpha,\beta)}.
\]
\end{conjecture}

Ahmed, Laforgia, and Muldoon~\cite[Theorem~3.1(ii)]
{AhmedLaforgiaMuldoon} had already proved the reverse inequality
throughout the square $|\alpha|,|\beta|\leq1/2$, apart from the four
equality cases; Gautschi later drew attention to this earlier
result~\cite{Gautschi2011}. Lun and Rafaeli~\cite{LunRafaeli}
subsequently used Sturm comparison to establish the inequality in
Conjecture~\ref{conj:first} on substantial parameter subregions and
treated Conjecture~\ref{conj:second}
in~\cite[Theorem~3]{LunRafaeli}. After the Liouville transformation,
the relevant solution is asymptotic to a constant multiple of
$\theta^{\alpha+1/2}$; it
therefore diverges when $-1<\alpha<-1/2$, while for
$\alpha=-1/2$ it tends to a nonzero constant. In their proof of
Theorem~3~\cite[p.~562]{LunRafaeli}, verification of the endpoint
Wronskian condition is reduced to leading asymptotic equivalences.
Those equivalences alone do not provide the derivative control needed
to establish the cancellation, and the cases $\alpha=\pm1/2$ are
not treated there. Lemma~\ref{lem:wronskian} supplies the required
endpoint verification throughout Gautschi's conjectured domain.

Our method also replaces the two particular scales $n$ and
$\gamma_n$ by the affine scale $n+\sigma$, where $\sigma$ may depend
on the Jacobi parameters but not on the degree. Both conjectures then
become consequences of a single differential comparison. The
resulting parameter-dependent thresholds yield ranges in which the
scaled zeros are monotone in the degree; they also characterise all
cases of equality and provide finite-degree bounds in terms of Bessel
zeros. These thresholds are optimal for the pointwise comparison of
the interpolating potentials; for general affine scales, they are not
asserted to be necessary for inequalities between individual zeros.
Corollary~\ref{cor:increasing-affine-classification} nevertheless shows
that $|\beta|\leq1/2$ is the exact condition for the existence of an
increasing affine comparison with a degree-independent shift. For the
spectral scale $n+\rho$,
Proposition~\ref{prop:spectral-failure} and
Corollary~\ref{cor:spectral-classification} show that the two spectral
parameter regions below are exact.

The proof begins in Section~\ref{sec:normal-form} by placing the
Jacobi equation in Liouville normal form and rescaling its independent
variable. The singular endpoint is treated there without imposing
boundedness of the transformed solution.
Section~\ref{sec:trigonometric} identifies the elementary profile
governing the derivative of the rescaled potential and obtains a
sharp criterion for its sign. Sturm comparison then yields the
increasing and decreasing affine inequalities in
Sections~\ref{sec:increasing} and~\ref{sec:decreasing}.
Section~\ref{sec:spectral-sharpness} combines the endpoint comparison
with the Mehler--Heine limit to prove the sharpness of the increasing
affine strip and of the spectral parameter regions.
Section~\ref{sec:bessel-bounds} records the
resulting bounds in terms of Bessel zeros, and the concluding section
summarises the parameter classification.

\section{Main results}
\label{sec:main-results}

The two principal results give complementary comparison principles
for affine rescalings of the angular zeros of Jacobi polynomials. The
first identifies scales for which the rescaled zeros are
nondecreasing with the degree, while the second identifies those for
which they are nonincreasing.

\begin{theorem}\label{thm:increasing}
Let $\alpha,\beta>-1$ with $|\beta|\leq1/2$. For each integer
$j\geq1$, let
\[
0<
\theta_{j,1}^{(\alpha,\beta)}
<
\cdots
<
\theta_{j,j}^{(\alpha,\beta)}
<\pi
\]
be determined by
\[
P_j^{(\alpha,\beta)}
\bigl(\cos\theta_{j,k}^{(\alpha,\beta)}\bigr)=0,
\quad
1\leq k\leq j.
\]
Set
\[
\rho\coloneqq\frac{\alpha+\beta+1}{2}
\]
and define
\[
\Phi_{\alpha,\beta}(x)
\coloneqq
\frac{1-4\alpha^2}{16}
\frac{1-x\cot x}{\sin^2x}
+
\frac{1-4\beta^2}{16}
\frac{1+x\tan x}{\cos^2x},
\quad
0<x<\frac{\pi}{2}.
\]
Writing $s_+\coloneqq\max\{s,0\}$, put
\[
m_{\alpha,\beta}
\coloneqq
\inf_{0<x<\pi/2}\Phi_{\alpha,\beta}(x).
\]
This quantity is finite. Define
\begin{equation}
\sigma_\uparrow(\alpha,\beta)
\coloneqq
\rho-
\frac{(-m_{\alpha,\beta})_+}{1+\rho}.
\label{eq:sigma-up}
\end{equation}
If $-1<\sigma\leq\sigma_\uparrow(\alpha,\beta)$, then
\begin{equation}
(n+\sigma)\theta_{n,k}^{(\alpha,\beta)}
\leq
(n+1+\sigma)\theta_{n+1,k}^{(\alpha,\beta)},
\label{eq:increasing-intro}
\end{equation}
for every integer $n\geq1$ and $1\leq k\leq n$. Equality holds for
some such $n$ and $k$ if and only if
\[
|\alpha|=|\beta|=\frac12, \quad
\sigma=\rho.
\]
In that case, equality holds for every integer $n\geq1$ and
$1\leq k\leq n$.
\end{theorem}

\begin{theorem}\label{thm:decreasing}
Let $\alpha,\beta>-1$ with $|\beta|\geq1/2$, and define
$\theta_{j,k}^{(\alpha,\beta)}$, $\rho$, $\Phi_{\alpha,\beta}$, and
$s_+$ as in Theorem~\ref{thm:increasing}. Put
\[
M_{\alpha,\beta}
\coloneqq
\sup_{0<x<\pi/2}\Phi_{\alpha,\beta}(x).
\]
This quantity is finite. Define
\begin{equation}
\sigma_\downarrow(\alpha,\beta)
\coloneqq
\rho+
\frac{(M_{\alpha,\beta})_+}{1+\rho}.
\label{eq:sigma-down}
\end{equation}
If $\sigma\geq\sigma_\downarrow(\alpha,\beta)$, then
\begin{equation}
(n+1+\sigma)\theta_{n+1,k}^{(\alpha,\beta)}
\leq
(n+\sigma)\theta_{n,k}^{(\alpha,\beta)},
\label{eq:decreasing-intro}
\end{equation}
for every integer $n\geq1$ and $1\leq k\leq n$. Equality holds for
some such $n$ and $k$ if and only if
\[
|\alpha|=|\beta|=\frac12, \quad
\sigma=\rho.
\]
In that case, equality holds for every integer $n\geq1$ and
$1\leq k\leq n$.
\end{theorem}

The thresholds in~\eqref{eq:sigma-up} and~\eqref{eq:sigma-down}
involve only the infimum or supremum of an elementary function of one
variable. For the increasing theorem, the following explicit
algebraic lower bound is often more convenient. Define
\begin{equation}
\sigma_{\mathrm{alg}}(\alpha,\beta)
\coloneqq
\rho-
\max\left\{
0,\,
\frac{\alpha^2+3\beta^2-1}{4(1+\rho)}
\right\}.
\label{eq:sigma-alg}
\end{equation}
Lemma~\ref{lem:explicit-threshold} shows that
\begin{equation}
\sigma_\uparrow(\alpha,\beta)
\geq
\sigma_{\mathrm{alg}}(\alpha,\beta)
>-\frac12.
\label{eq:sigma-comparison}
\end{equation}
In particular, Theorem~\ref{thm:increasing} applies throughout the
explicit interval
\[
-1<\sigma\leq
\sigma_{\mathrm{alg}}(\alpha,\beta).
\]
On $\mathcal D_1$, the right-hand endpoint is nonnegative, so
Theorem~\ref{thm:increasing} may be applied with $\sigma=0$. Its
equality criterion shows that equality can then occur only at
\[
\left(-\frac12,-\frac12\right),
\]
which is excluded from $\mathcal D_1$. This proves
Conjecture~\ref{conj:first}.

Whenever $\sigma_{\mathrm{alg}}(\alpha,\beta)>0$, every choice
$0<\sigma\leq\sigma_{\mathrm{alg}}(\alpha,\beta)$ yields a strictly
sharper comparison factor than the choice $\sigma=0$. The proof of Corollary~\ref{cor:first-conjecture} also shows that
$\sigma_{\mathrm{alg}}(\alpha,\beta)>0$ whenever
$(\alpha,\beta)\in\mathcal D_1$ and
$\alpha+\beta+1>0$. Positive shifts therefore give strictly sharper
comparison factors throughout this part of $\mathcal D_1$. On the
boundary $\alpha+\beta+1=0$, one has
$\sigma_\uparrow(\alpha,\beta)=0$, so no positive shift is
available. If
\[
|\beta|\leq\frac12, \quad
\alpha+\beta+1<0,
\]
then $\sigma_\uparrow(\alpha,\beta)\leq\rho<0$, whereas
\eqref{eq:sigma-comparison} ensures that
$\sigma_\uparrow(\alpha,\beta)>-1/2$.
Theorem~\ref{thm:increasing} therefore still gives an increasing
affine comparison for every
\[
-1<\sigma\leq\sigma_\uparrow(\alpha,\beta).
\]
This triangular region is an additional range for the affine
comparison theorem, rather than an enlargement of Gautschi's
unshifted conjecture. Its geometry and its relation to
$\mathcal D_1$ are shown in
Figure~\ref{fig:increasing-regions}.

The condition $|\beta|\leq1/2$ is also necessary for the existence of
an increasing affine comparison valid for all zero indices.

\begin{corollary}\label{cor:increasing-affine-classification}
Let $\alpha,\beta>-1$. There exists a real number $\sigma>-1$ such that
\[
(n+\sigma)\theta_{n,k}^{(\alpha,\beta)}
\leq
(n+1+\sigma)\theta_{n+1,k}^{(\alpha,\beta)}
\]
for every $n\geq1$ and $1\leq k\leq n$ if and only if
\[
|\beta|\leq\frac12.
\]
\end{corollary}

The proof is given in Section~\ref{sec:spectral-sharpness}.

\begin{figure}[htbp]
\centering
\begin{tikzpicture}[
  x=2.8cm,
  y=2.8cm,
  line cap=round,
  line join=round
]
  \tikzset{
    affineextension/.style={
      fill=black!9,
      draw=none
    },
    original/.style={
      fill=black!18,
      draw=none
    },
    boundary/.style={
      draw=black!60,
      line width=0.72pt
    },
    excluded/.style={
      draw=black!52,
      line width=0.64pt,
      dash pattern=on 3pt off 2.2pt
    },
    axis/.style={
      draw=black!72,
      line width=0.68pt,
      -stealth
    },
    regionlabel/.style={
      font=\scriptsize,
      text=black!78,
      align=center,
      inner sep=0pt
    },
    smalllabel/.style={
      font=\scriptsize,
      text=black!68,
      inner sep=1.2pt
    }
  }

  \path[use as bounding box]
    (-1.30,-0.88) rectangle (2.22,0.91);

  \fill[affineextension]
    (-1,-0.5) rectangle (2.08,0.5);

  \fill[original]
    (-0.5,-0.5)
    -- (2.08,-0.5)
    -- (2.08,0.5)
    -- (-1,0.5)
    -- (-1,0)
    -- cycle;

  \fill[white,path fading=west]
    (1.53,-0.5) rectangle (2.08,0.5);

  \draw[boundary]
    (-1,-0.5) -- (2.08,-0.5);

  \draw[boundary]
    (-1,0.5) -- (2.08,0.5);

  \draw[boundary]
    (-1,0) -- (-0.5,-0.5);

  \draw[excluded]
    (-1,-0.5) -- (-1,0.5);

  \draw[axis]
    (-1.12,0) -- (2.15,0)
    node[right,smalllabel] {$\alpha$};

  \draw[axis]
    (0,-0.72) -- (0,0.79)
    node[above,smalllabel] {$\beta$};

  \foreach \x/\lab in {
    -1/{-1},
    -0.5/{-\frac12},
    0.5/{\frac12},
    1/{1},
    2/{2}
  }
  {
    \draw[black!62,line width=0.5pt]
      (\x,0.025) -- (\x,-0.025);

    \node[smalllabel,below]
      at (\x,-0.035) {$\lab$};
  }

  \foreach \y/\lab in {
    -0.5/{-\frac12},
    0.5/{\frac12}
  }
  {
    \draw[black!62,line width=0.5pt]
      (0.025,\y) -- (-0.025,\y);

    \node[smalllabel,left]
      at (-0.035,\y) {$\lab$};
  }

  \node[regionlabel]
    at (0.86,0.19)
    {Gautschi's original domain $\mathcal D_1$};

  \node[regionlabel,anchor=east]
    at (-1.08,-0.28)
    {$\sigma_{\uparrow}<0$};

  \draw[black!54,line width=0.48pt,-stealth]
    (-1.06,-0.29) -- (-0.82,-0.33);

  \draw[black!54,line width=0.48pt,-stealth]
    (-0.36,-0.30) -- (-0.70,-0.30);

  \node[smalllabel,anchor=west]
    at (-0.34,-0.30)
    {$\alpha+\beta+1=0$};

  \filldraw[
    fill=white,
    draw=black!72,
    line width=0.62pt
  ]
    (-0.5,-0.5) circle (1.5pt);

\end{tikzpicture}
\caption{Parameter regions for the increasing affine comparison.
The dark-grey set is Gautschi's original domain $\mathcal D_1$,
where the choice $\sigma=0$ gives the inequality in
Conjecture~\ref{conj:first}. The light-grey triangle is the
additional range in which
$\sigma_{\uparrow}(\alpha,\beta)<0$; there
Theorem~\ref{thm:increasing} provides an increasing affine
comparison for
$-1<\sigma\leq\sigma_{\uparrow}(\alpha,\beta)$,
and hence only for negative admissible shifts. Where
$\alpha+\beta+1>0$ in $\mathcal D_1$, positive shifts yield strictly
sharper comparison factors. The sloping line marks
$\alpha+\beta+1=0$, where $\sigma_{\uparrow}=0$; its admissible
portion belongs to $\mathcal D_1$, except for the open point
$(-1/2,-1/2)$, at which the unshifted comparison is an equality.
Together with the marked point $(-1/2,-1/2)$, the shaded set is exactly
the parameter region in which some degree-independent affine shift gives
an increasing ordering for all degrees and zero indices; by
Corollary~\ref{cor:increasing-affine-classification}, no such shift
exists in the unshaded region. The solid horizontal lines mark the
included boundaries $|\beta|=1/2$, the vertical dashed line marks the
excluded boundary $\alpha=-1$, and the fading indicates that the shaded
region is unbounded in the positive $\alpha$-direction.}
\label{fig:increasing-regions}
\end{figure}
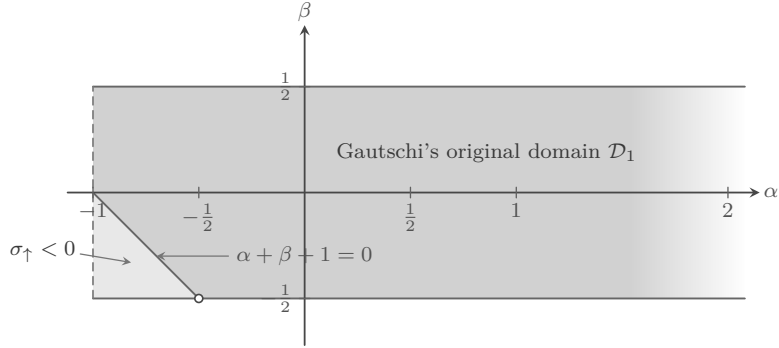

For the spectral scale $h_n=n+\rho$, it is convenient to write
\begin{equation}
A\coloneqq\frac{1-4\alpha^2}{16},
\quad
B\coloneqq\frac{1-4\beta^2}{16},
\label{eq:AB}
\end{equation}
and
\begin{equation}
R(x)\coloneqq\frac{1-x\cot x}{\sin^2x},
\quad
T(x)\coloneqq\frac{1+x\tan x}{\cos^2x}.
\label{eq:RT}
\end{equation}
Thus $\Phi_{\alpha,\beta}=AR+BT$. Introduce the two regions
\begin{align*}
\mathcal D_\downarrow
&\coloneqq
\left\{
(\alpha,\beta)\in(-1,\infty)^2:
|\beta|\geq\frac12,\,\,
\alpha^2+3\beta^2\geq1
\right\}\\[7pt]
\mathcal D_\uparrow
&\coloneqq
\left\{
(\alpha,\beta)\in(-1,\infty)^2:
|\beta|\leq\frac12,\,\,
\alpha^2+3\beta^2\leq1
\right\}.
\end{align*}

The ellipse $\alpha^2+3\beta^2=1$ already arose as a fixed-index,
large-degree asymptotic boundary in Gautschi's
analysis~\cite[Section~2]{Gautschi2009b}, which used Gatteschi's
expansion for Jacobi zeros~\cite[Theorem~4.1]{Gatteschi1985}. The
point established here is that, together with the two lines
$\beta=\pm1/2$, it gives the exact sign classification of the
spectral-scale profile. Proposition~\ref{prop:spectral-sign} shows
that $\Phi_{\alpha,\beta}$ is nonpositive throughout $(0,\pi/2)$
precisely on $\mathcal D_\downarrow$, and nonnegative throughout that
interval precisely on $\mathcal D_\uparrow$.

On $\mathcal D_\downarrow$ one has
$\sigma_\downarrow(\alpha,\beta)=\rho$. Hence
Theorem~\ref{thm:decreasing} with $\sigma=\rho$ establishes the
inequality appearing in Conjecture~\ref{conj:second} throughout
$\mathcal D_\downarrow$; the inequality is strict except at the four
points
\[
|\alpha|=|\beta|=\frac12,
\]
where equality holds. Since
$\mathcal D_2\subset\mathcal D_\downarrow$ and these four points are
excluded from $\mathcal D_2$, this proves
Conjecture~\ref{conj:second} and establishes the same inequality on
the larger set $\mathcal D_\downarrow$. For instance,
\[
(0,1)\in\mathcal D_\downarrow\setminus\mathcal D_2.
\]
The additional upper and lower strips with
$-1/2<\alpha<1/2$ were already identified computationally by
Gautschi~\cite[p.~296]{Gautschi2009b}; the present theorem proves the
ordering throughout them.
On $\mathcal D_\uparrow$ one has
$\sigma_\uparrow(\alpha,\beta)=\rho$, and the reverse spectral
comparison holds; it is strict except at the same four points.
Thus the reverse ordering proved by Ahmed, Laforgia, and
Muldoon~\cite[Theorem~3.1(ii)]{AhmedLaforgiaMuldoon} holds beyond the
square $|\alpha|,|\beta|\leq1/2$.
Indeed, that square is contained in $\mathcal D_\uparrow$, since
\[
\alpha^2+3\beta^2\leq
\frac14+\frac34=1,
\]
and the inclusion is strict; for example,
$(3/4,0)\in\mathcal D_\uparrow$ lies outside the square. This is an
extension of the {reverse} ordering, not an enlargement of the
decreasing inequality in Conjecture~\ref{conj:second} into the
ellipse.
Gautschi had already conjectured the non-strict reverse ordering in
these lateral regions within the ellipse on the basis of numerical
experiments~\cite[Conjecture~2]{Gautschi2009b}; the result above
provides a proof and determines exactly the parameter region on which
the profile is nonnegative.

The set $\mathcal D_\uparrow$ is not the whole interior of the
ellipse. Indeed, put
\[
C\coloneqq A+3B=\frac{1-\alpha^2-3\beta^2}{4}.
\]
In the open ellipse one has $C>0$. In its central band
$|\beta|\leq1/2$, one also has $B\geq0$, and therefore
\[
\Phi_{\alpha,\beta}(x)
=
C R(x)+B\bigl(T(x)-3R(x)\bigr)>0,
\quad 0<x<\frac{\pi}{2}.
\]
This is precisely $\mathcal D_\uparrow$ intersected with the open
ellipse. By contrast, each of the two caps cut out by
\[
\frac12<|\beta|<\frac1{\sqrt3},
\quad
\alpha^2+3\beta^2<1,
\]
has $B<0$ and $C>0$. In either cap,
\[
\lim_{x\downarrow0}\Phi_{\alpha,\beta}(x)=\frac{C}{3}>0,
\quad
\lim_{x\uparrow\pi/2}\Phi_{\alpha,\beta}(x)=-\infty,
\]
so the profile changes sign. The same occurs in the two lateral
regions outside the ellipse with $|\beta|<1/2$, where $B>0$ and $C<0$.
Thus the unshaded region in Figure~\ref{fig:spectral-regions} is
\[
\mathcal U
\coloneqq
\left\{
(\alpha,\beta)\in(-1,\infty)^2:
BC<0
\right\}
=
(-1,\infty)^2
\setminus
\bigl(\mathcal D_\uparrow\cup\mathcal D_\downarrow\bigr).
\]
On the elliptic boundary $C=0$, the sign of
$\Phi_{\alpha,\beta}=B(T-3R)$ is the sign of $B$, except at the four
intersections with $|\beta|=1/2$, where it vanishes identically.
Gautschi had already reported numerical examples of both directions
of the inequality in the upper and lower cap
regions~\cite[p.~295]{Gautschi2009b}.

\begin{proposition}\label{prop:spectral-failure}
Let $(\alpha,\beta)\in\mathcal U$, and put
\[
\Delta_{n,k}
\coloneqq
\gamma_{n+1}\theta_{n+1,k}^{(\alpha,\beta)}
-
\gamma_n\theta_{n,k}^{(\alpha,\beta)}.
\]
There exists $N=N(\alpha,\beta)$ such that, for every $n\geq N$,
$\Delta_{n,1}$ has the same sign as $C$, whereas
$\Delta_{n,n}$ has the same sign as $B$.
In particular,
\[
\Delta_{n,1}\Delta_{n,n}<0,
\quad n\geq N.
\]
\end{proposition}

Thus no spectral ordering valid for all zero indices can hold at a
parameter point in $\mathcal U$. The proof of
Proposition~\ref{prop:spectral-failure} is given in
Section~\ref{sec:spectral-sharpness}.

\begin{corollary}\label{cor:spectral-classification}
Let $\alpha,\beta>-1$. The inequality
\[
\gamma_n\theta_{n,k}^{(\alpha,\beta)}
\leq
\gamma_{n+1}\theta_{n+1,k}^{(\alpha,\beta)}
\]
holds for every $n\geq1$ and $1\leq k\leq n$ if and only if
$(\alpha,\beta)\in\mathcal D_\uparrow$. Likewise,
\[
\gamma_{n+1}\theta_{n+1,k}^{(\alpha,\beta)}
\leq
\gamma_n\theta_{n,k}^{(\alpha,\beta)}
\]
holds for every $n\geq1$ and $1\leq k\leq n$ if and only if
$(\alpha,\beta)\in\mathcal D_\downarrow$.
\end{corollary}

The proof of Corollary~\ref{cor:spectral-classification} is also
given in Section~\ref{sec:spectral-sharpness}. The geometry of the
two ordered regions, and their relation to Gautschi's original
domain, is shown in Figure~\ref{fig:spectral-regions}.

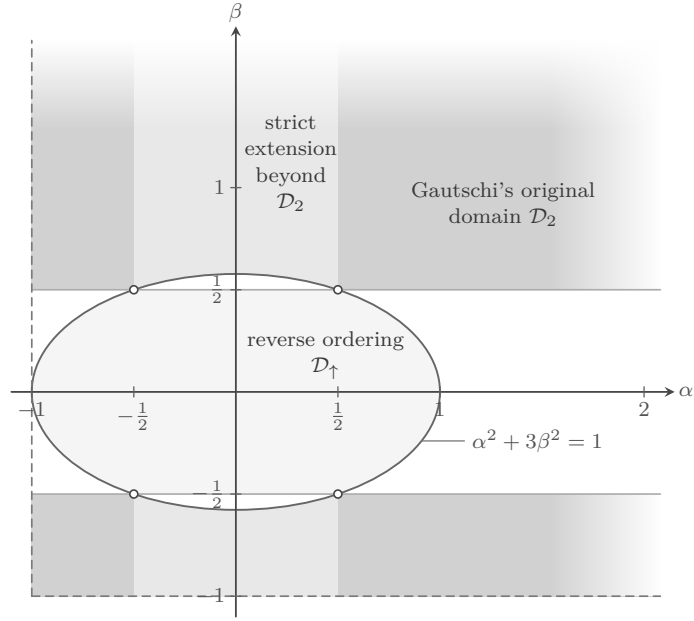
\begin{figure}[htbp]
\centering
\begin{tikzpicture}[
  x=2.7cm,
  y=2.7cm,
  line cap=round,
  line join=round
]
  \tikzset{
    decreasing/.style={
      fill=black!9,
      draw=none
    },
    increasing/.style={
      fill=black!4,
      draw=none
    },
    original/.style={
      fill=black!18,
      draw=none
    },
    boundary/.style={
      draw=black!60,
      line width=0.72pt
    },
    divider/.style={
      draw=black!32,
      line width=0.64pt
    },
    excluded/.style={
      draw=black!52,
      line width=0.64pt,
      dash pattern=on 3pt off 2.2pt
    },
    axis/.style={
      draw=black!72,
      line width=0.68pt,
      -stealth
    },
    regionlabel/.style={
      font=\scriptsize,
      text=black!78,
      align=center,
      inner sep=0pt
    },
    smalllabel/.style={
      font=\scriptsize,
      text=black!68,
      inner sep=1.2pt
    }
  }

  \path[use as bounding box]
    (-1.18,-1.17) rectangle (2.22,1.83);

  \fill[decreasing]
    (-1,-1) rectangle (2.08,-0.5);

  \fill[decreasing]
    (-1,0.5) rectangle (2.08,1.72);

  \fill[white]
    (0,0) ellipse [x radius=1,y radius=0.5773503];

  \begin{scope}
    \clip (-1,-0.5) rectangle (2.08,0.5);
    \fill[increasing]
      (0,0) ellipse [x radius=1,y radius=0.5773503];
  \end{scope}

  \fill[original]
    (-1,-1) rectangle (-0.5,-0.5);

  \fill[original]
    (0.5,-1) rectangle (2.08,-0.5);

  \fill[original]
    (-1,0.5) rectangle (-0.5,1.72);

  \fill[original]
    (0.5,0.5) rectangle (2.08,1.72);

  \fill[white,path fading=west]
    (1.53,-1) rectangle (2.08,1.72);

  \fill[white,path fading=south]
    (-1,1.28) rectangle (2.08,1.72);

  \draw[boundary]
    (0,0) ellipse [x radius=1,y radius=0.5773503];

  \draw[divider]
    (-1,-0.5) -- (2.08,-0.5);

  \draw[divider]
    (-1,0.5) -- (2.08,0.5);

  \draw[excluded]
    (-1,-1) -- (2.08,-1);

  \draw[excluded]
    (-1,-1) -- (-1,1.72);

  \draw[axis]
    (-1.10,0) -- (2.15,0)
    node[right,smalllabel] {$\alpha$};

  \draw[axis]
    (0,-1.10) -- (0,1.79)
    node[above,smalllabel] {$\beta$};

  \foreach \x/\lab in {
    -1/{-1},
    -0.5/{-\frac12},
    0.5/{\frac12},
    1/{1},
    2/{2}
  }
  {
    \draw[black!62,line width=0.5pt]
      (\x,0.025) -- (\x,-0.025);

    \node[smalllabel,below]
      at (\x,-0.035) {$\lab$};
  }

  \foreach \y/\lab in {
    -1/{-1},
    -0.5/{-\frac12},
    0.5/{\frac12},
    1/{1}
  }
  {
    \draw[black!62,line width=0.5pt]
      (0.025,\y) -- (-0.025,\y);

    \node[smalllabel,left]
      at (-0.035,\y) {$\lab$};
  }

  \node[regionlabel]
    at (0.44,0.18)
    {reverse ordering\\$\mathcal D_\uparrow$};

  \node[regionlabel]
    at (0.27,1.12)
    {strict\\extension\\beyond\\$\mathcal D_2$};

  \node[regionlabel]
    at (1.31,0.92)
    {Gautschi's original\\domain $\mathcal D_2$};

  \draw[black!54,line width=0.48pt]
    (0.91,-0.24) -- (1.12,-0.24);

  \node[smalllabel,anchor=west]
    at (1.14,-0.24)
    {$\alpha^2+3\beta^2=1$};

  \foreach \x in {-0.5,0.5}
    \foreach \y in {-0.5,0.5}
      \filldraw[
        fill=white,
        draw=black!72,
        line width=0.62pt
      ]
      (\x,\y) circle (1.5pt);

\end{tikzpicture}
\caption{Parameter regions for the spectral scale
$n+(\alpha+\beta+1)/2$. The dark-grey set is Gautschi's original
domain $\mathcal D_2$, while the medium-grey set is the additional
range covered by the present theorem; together with the four marked
points, they form $\mathcal D_\downarrow$. In the light-grey set
$\mathcal D_\uparrow$ the reverse inequality holds, and the marked
points are precisely the equality cases. In the unshaded set
$\mathcal U$, Proposition~\ref{prop:spectral-failure} shows that no
uniform spectral ordering holds. The inner solid lines mark
$|\beta|=1/2$, the outer dashed lines mark the excluded boundaries,
and the fading indicates that the shaded regions are unbounded.}
\label{fig:spectral-regions}
\end{figure}

\section{Liouville normal form and zero comparison}
\label{sec:normal-form}

Henceforth, the parameter superscript will be suppressed whenever no
ambiguity can arise; thus
\[
\theta_{n,k}\coloneqq\theta_{n,k}^{(\alpha,\beta)}.
\]

We begin with the Jacobi differential equation
\begin{equation}
(1-x^2)y''(x)
+
\bigl[\beta-\alpha-(\alpha+\beta+2)x\bigr]y'(x)
+
n(n+\alpha+\beta+1)y(x)
=0.
\label{eq:jacobi-equation}
\end{equation}
After the change of variables $x=\cos\theta$, introduce
\begin{equation}
Y_n(\theta)
\coloneqq
\left(\sin\frac{\theta}{2}\right)^{\alpha+1/2}
\left(\cos\frac{\theta}{2}\right)^{\beta+1/2}
P_n^{(\alpha,\beta)}(\cos\theta).
\label{eq:liouville-solution}
\end{equation}
A direct calculation from~\eqref{eq:jacobi-equation} gives
\begin{equation}
Y_n''(\theta)
+
\left(
\left(n+\rho\right)^2
+
\frac{A}{\sin^2(\theta/2)}
+
\frac{B}{\cos^2(\theta/2)}
\right)Y_n(\theta)
=0,
\label{eq:liouville-equation}
\end{equation}
where $\rho=(\alpha+\beta+1)/2$ and $A$ and $B$ are defined in~\eqref{eq:AB}.
This normal form is standard; see, for
example,~\cite[equation~(4.24.2)]{Szego}. For an integer $m\geq1$
and a number $h>0$, define
\begin{equation}
Z_{m,h}(t)\coloneqq Y_m(t/h),
\quad
0<t<h\pi.
\label{eq:scaled-solution}
\end{equation}
Its zeros are precisely
\begin{equation}
h\theta_{m,1}
<\cdots<
h\theta_{m,m}.
\label{eq:scaled-zeros}
\end{equation}
Moreover,
\begin{equation}
Z_{m,h}''(t)
+
\left(
\left(\frac{m+\rho}{h}\right)^2
+
\frac{A}{h^2\sin^2(t/2h)}
+
\frac{B}{h^2\cos^2(t/2h)}
\right)Z_{m,h}(t)
=0.
\label{eq:scaled-equation}
\end{equation}

At the left endpoint $t=0$, which corresponds to $\theta=0$ and
$x=1$, the only term in the coefficient of $Z_{m,h}$
in~\eqref{eq:scaled-equation} that may fail to remain finite is
\[
\frac{A}{h^2\sin^2(t/2h)}
=
\frac{1-4\alpha^2}{4t^2}+O(1),
\quad t\downarrow0.
\]
Consequently, $t=0$ is a singular endpoint of the normal-form
equation when $|\alpha|\ne1/2$. If $|\alpha|=1/2$, then $A=0$ and
every coefficient in~\eqref{eq:scaled-equation} extends continuously
to $t=0$. The following lemma supplies the precise boundary condition
required for the comparison argument in both cases.

\begin{lemma}\label{lem:wronskian}
Let $m,\ell\geq1$ be integers and let $h,g>0$. Then
\begin{equation}
\lim_{t\downarrow0}
\left(
Z_{m,h}'(t)Z_{\ell,g}(t)
-
Z_{m,h}(t)Z_{\ell,g}'(t)
\right)
=0.
\label{eq:wronskian-limit}
\end{equation}
\end{lemma}

\begin{proof}
Since $P_m^{(\alpha,\beta)}(1)>0$, the expansions of
$\sin(t/2h)$, $\cos(t/2h)$, and
$P_m^{(\alpha,\beta)}(\cos(t/h))$ at $t=0$ give
\begin{equation}
Z_{m,h}(t)
=
C_{m,h}t^{\alpha+1/2}
\bigl(1+c_{m,h}t^2+O(t^4)\bigr),
\label{eq:endpoint-expansion}
\end{equation}
where
\[
C_{m,h}
=
(2h)^{-\alpha-1/2}P_m^{(\alpha,\beta)}(1)
>0.
\]
The factor multiplying $t^{\alpha+1/2}$ is analytic in $t^2$, so
the expansion may be differentiated. The same statement holds for
$Z_{\ell,g}$, with different constants. Writing
$s=\alpha+1/2$, direct substitution gives
\[
Z_{m,h}'Z_{\ell,g}-Z_{m,h}Z_{\ell,g}'
=
2C_{m,h}C_{\ell,g}
\bigl(c_{m,h}-c_{\ell,g}\bigr)t^{2s+1}
+
O\left(t^{2s+3}\right).
\]
Thus the leading terms cancel, and the expression in~\eqref{eq:wronskian-limit} is
\[
O\left(t^{2\alpha+2}\right).
\]
Because $\alpha>-1$, this tends to zero. The argument also covers
$-1<\alpha<-1/2$, where the individual solutions are unbounded at
the endpoint.
\end{proof}

\begin{remark}\label{rem:singular-endpoint}
When $\alpha\geq-1/2$, the transformed solution is bounded at the
left endpoint; when $-1<\alpha<-1/2$, it is not. In the latter range,
assigning a finite endpoint value at $t=0$ would therefore be
invalid. The requisite condition is instead the vanishing of the
boundary Wronskian; the cancellation
in~\eqref{eq:wronskian-limit} is therefore essential, not merely a
technical regularity detail.
\end{remark}

We shall use the following version of Sturm's comparison theorem,
formulated so that the singular left endpoint is covered by the
Wronskian condition. For the classical comparison theorem,
see~\cite[Section~5.5]{Teschl}; related applications to Jacobi zeros
may be found
in~\cite{AhmedLaforgiaMuldoon,DeanoGilSegura,LunRafaeli}.

\begin{lemma}[Sturm comparison]\label{lem:sturm}
Let $q_1,q_2$ be continuous on $(0,L)$ and satisfy
\[
q_1(t)>q_2(t),
\quad
0<t<L.
\]
Let $u_1,u_2$ be nontrivial solutions of
\[
u_j''+q_j u_j=0,
\quad
j=1,2,
\]
which are both positive in some right neighbourhood of $0$, and
suppose that
\[
\lim_{t\downarrow0}
\bigl(u_1'u_2-u_1u_2'\bigr)=0.
\]
If $b_k<L$ is the $k$th positive zero of $u_2$, then $u_1$ has at
least $k$ zeros in $(0,b_k)$. In particular, if $a_k$ denotes the
$k$th positive zero of $u_1$, then
\[
a_k<b_k.
\]
\end{lemma}

\begin{proof}
Let
\[
W(t)\coloneqq u_1'(t)u_2(t)-u_1(t)u_2'(t).
\]
Then
\begin{equation}
W'(t)=\bigl(q_2(t)-q_1(t)\bigr)u_1(t)u_2(t).
\label{eq:sturm-wronskian-derivative}
\end{equation}
We first prove that $u_1$ has a zero in $(0,b_1)$. Both solutions are
positive near $0$, and $u_2$ has no zero in $(0,b_1)$; hence $u_2$
is positive throughout that interval and $u_2'(b_1)<0$. Suppose, to
the contrary, that $u_1$ has no zero in $(0,b_1)$. It is then
positive there. For $0<\varepsilon<t<b_1$, integration
of~\eqref{eq:sturm-wronskian-derivative} gives
\[
W(t)-W(\varepsilon)
=
\int_\varepsilon^t
\bigl(q_2(s)-q_1(s)\bigr)u_1(s)u_2(s)\,ds<0.
\]
The assumed endpoint limit permits $\varepsilon\downarrow0$ and
yields $W(t)<0$ for $0<t<b_1$: strictness follows by integrating
over any compact subinterval of $(0,t)$, on which the continuous
integrand is strictly negative. Moreover, $W$ is
strictly decreasing. At the regular point $b_1$, however, continuity
gives
\[
W(b_1)=-u_1(b_1)u_2'(b_1)\geq0,
\]
including the case $u_1(b_1)=0$. This is a contradiction. Thus
$u_1$ has a zero in $(0,b_1)$.

For $j=2,\ldots,k$, the points $b_{j-1}$ and $b_j$ are regular
consecutive zeros of $u_2$. Suppose that $u_1$ had no zero between
them. After changing signs if necessary, both solutions would be
positive on $(b_{j-1},b_j)$. Hence
\[
u_2'(b_{j-1})>0,
\quad
u_2'(b_j)<0,
\]
and continuity would give
\[
W(b_{j-1})\leq0,
\quad
W(b_j)\geq0.
\]
This contradicts $W'<0$ throughout the interval. Thus $u_1$ has a
zero in $(b_{j-1},b_j)$. It follows that $u_1$ has at least one zero
in each of the disjoint intervals
\[
(0,b_1),(b_1,b_2),\ldots,(b_{k-1},b_k).
\]
Thus $u_1$ has at least $k$ zeros in $(0,b_k)$, and therefore
$a_k<b_k$.
\end{proof}

\section{The trigonometric comparison}
\label{sec:trigonometric}

Recall the functions $R$ and $T$ defined in~\eqref{eq:RT}.
With $A$ and $B$ as in~\eqref{eq:AB}, the profile introduced in
Theorems~\ref{thm:increasing} and~\ref{thm:decreasing} is
\begin{equation}
\Phi_{\alpha,\beta}(x)
=
A R(x)+B T(x).
\label{eq:Phi-RT}
\end{equation}

\begin{lemma}\label{lem:RT}
For every $0<x<\pi/2$,
\begin{equation}
0<R(x)<1
\label{eq:R-bounds}
\end{equation}
and
\begin{equation}
T(x)>3R(x).
\label{eq:T3R}
\end{equation}
Moreover,
\begin{equation}
\lim_{x\downarrow0}R(x)=\frac13,
\quad
\lim_{x\downarrow0}T(x)=1,
\label{eq:RT-zero}
\end{equation}
and
\begin{equation}
\lim_{x\uparrow\pi/2}R(x)=1,
\quad
\lim_{x\uparrow\pi/2}T(x)=+\infty.
\label{eq:RT-pi}
\end{equation}
\end{lemma}

\begin{proof}
The inequality $\tan x>x$ gives $1-x\cot x>0$, and hence
$R(x)>0$. Furthermore,
\[
1-x\cot x<\sin^2x
\]
is equivalent to
\[
\sin x\cos x<x,
\]
which follows from $\sin(2x)<2x$. This
proves~\eqref{eq:R-bounds}. To prove~\eqref{eq:T3R}, put $y=\tan x$.
A direct simplification shows that the desired inequality is
equivalent to
\begin{equation}
y^3(1+xy)>3(y-x).
\label{eq:y-inequality}
\end{equation}
Define
\[
F(x)
\coloneqq
\frac13\tan^3x\bigl(1+x\tan x\bigr)-\tan x+x.
\]
Extend $F$ continuously to the origin by setting $F(0)=0$.
Differentiation gives
\[
F'(x)
=
\frac43\tan^4x
+
\frac43x\tan^3x\sec^2x
>0.
\]
Thus $F(x)>0$ for $0<x<\pi/2$, which is
precisely~\eqref{eq:y-inequality}. The limits
in~\eqref{eq:RT-zero} follow from the Taylor expansions of $\cot x$,
$\tan x$, $\sin x$, and $\cos x$ at the origin. The limits
in~\eqref{eq:RT-pi} follow directly from the definitions.
\end{proof}

It follows in particular that
\[
\lim_{x\downarrow0}\Phi_{\alpha,\beta}(x)
=
\frac{A+3B}{3},
\]
whereas
\[
\lim_{x\uparrow\pi/2}\Phi_{\alpha,\beta}(x)
=
\begin{cases}
+\infty,&B>0,\\[7pt]
A,&B=0,\\[7pt]
-\infty,&B<0
\end{cases}.
\]
Since $\Phi_{\alpha,\beta}$ is continuous on $(0,\pi/2)$, these
limits show that $m_{\alpha,\beta}$ is finite when $B\geq0$ and that
$M_{\alpha,\beta}$ is finite when $B\leq0$, as asserted in
Theorems~\ref{thm:increasing} and~\ref{thm:decreasing}.

\begin{proposition}\label{prop:spectral-sign}
Put $C\coloneqq A+3B$. Then
\[
\Phi_{\alpha,\beta}(x)\geq0,
\quad0<x<\frac{\pi}{2}
\]
if and only if $B\geq0$ and $C\geq0$. Likewise,
\[
\Phi_{\alpha,\beta}(x)\leq0,
\quad0<x<\frac{\pi}{2}
\]
if and only if $B\leq0$ and $C\leq0$. Moreover, $\Phi_{\alpha,\beta}$ vanishes identically precisely when
$B=C=0$, which, in turn, occurs precisely when
\[
|\alpha|=|\beta|=\frac12.
\]
\end{proposition}

\begin{proof}
The identity
\[
\Phi_{\alpha,\beta}(x)
=
C R(x)+B\bigl(T(x)-3R(x)\bigr)
\]
and Lemma~\ref{lem:RT} prove sufficiency in both sign assertions.
Conversely, the limit at the origin shows that a nonnegative profile
must have $C\geq0$, whereas a nonpositive profile must have $C\leq0$.
The limit at $\pi/2$ similarly excludes $B<0$ in the first case and
$B>0$ in the second. This proves both equivalences.

If the profile vanishes identically, the two sign assertions give
$B=C=0$; the converse follows from the displayed identity. Finally,
$B=0$ is equivalent to $|\beta|=1/2$, and then $C=0$ is equivalent
to $|\alpha|=1/2$.
\end{proof}

The following exact differentiation formula explains the role of
these functions in the comparison argument.

\begin{lemma}\label{lem:potential-derivative}
Fix $\sigma>-1$ and set $\delta\coloneqq\rho-\sigma$.
For $h>0$ and $0<t<h\pi$, define
\begin{equation}
Q_h(t)
\coloneqq
\left(1+\frac{\delta}{h}\right)^2
+
\frac{A}{h^2\sin^2(t/2h)}
+
\frac{B}{h^2\cos^2(t/2h)}.
\label{eq:Qh}
\end{equation}
If $x=t/(2h)$, then
\begin{equation}
\frac{\partial Q_h(t)}{\partial h}
=
-\frac{2}{h^3}
\left(
\delta(h+\delta)+A R(x)+B T(x)
\right).
\label{eq:Q-derivative}
\end{equation}
\end{lemma}

\begin{proof}
The first term in~\eqref{eq:Qh} contributes
\[
-\frac{2}{h^3}\delta(h+\delta).
\]
Since $x=t/(2h)$,
\begin{align*}
\frac{\partial}{\partial h}
\left(
\frac{1}{h^2\sin^2(t/2h)}
\right)
&=
-\frac{2}{h^3}
\frac{1-x\cot x}{\sin^2x},
\\[7pt]
\frac{\partial}{\partial h}
\left(
\frac{1}{h^2\cos^2(t/2h)}
\right)
&=
-\frac{2}{h^3}
\frac{1+x\tan x}{\cos^2x}.
\end{align*}
Combining these three identities proves~\eqref{eq:Q-derivative}.
\end{proof}

\begin{proposition}\label{prop:sharp-potential}
Let $\alpha,\beta>-1$ and $\sigma>-1$, and let $Q_h$ be defined
by~\eqref{eq:Qh}. Then the following assertions hold.
\begin{enumerate}
\item If $|\beta|>1/2$, the inequality
\[
\frac{\partial Q_h(t)}{\partial h}\leq0
\]
cannot hold for every $h\geq1+\sigma$ and $0<t<h\pi$. If
$|\beta|\leq1/2$, it holds throughout that domain if and only if
\[
\sigma\leq\sigma_\uparrow(\alpha,\beta).
\]
\item If $|\beta|<1/2$, the inequality
\[
\frac{\partial Q_h(t)}{\partial h}\geq0
\]
cannot hold for every $h\geq1+\sigma$ and $0<t<h\pi$. If
$|\beta|\geq1/2$, it holds throughout that domain if and only if
\[
\sigma\geq\sigma_\downarrow(\alpha,\beta).
\]
\end{enumerate}
In either case, for each fixed $t>0$, the function
$h\mapsto Q_h(t)$ is constant on a nondegenerate interval of its
admissible domain if and only if
\[
A=B=0,
\quad
\sigma=\rho.
\]
\end{proposition}

\begin{proof}
Put
\[
\delta\coloneqq\rho-\sigma,
\quad
\Phi\coloneqq\Phi_{\alpha,\beta}.
\]
Since $\alpha,\beta>-1$,
\[
1+\rho
=
\frac{\alpha+\beta+3}{2}
>
\frac12.
\]
By~\eqref{eq:Q-derivative}, the first inequality is equivalent to
\begin{equation}
\delta(h+\delta)+\Phi(x)\geq0,
\quad
h\geq1+\sigma,
\quad
0<x<\frac{\pi}{2}.
\label{eq:positive-bracket}
\end{equation}
Suppose that~\eqref{eq:positive-bracket} holds. Fix
$x\in(0,\pi/2)$. For each $h$, this value of $x$ corresponds to the
admissible point $t=2hx$. Letting $h\to\infty$ therefore shows that
$\delta\geq0$. For any fixed admissible $h$, letting
$x\uparrow\pi/2$ and using~\eqref{eq:RT-pi} gives $B\geq0$. At
$h=1+\sigma$ one has
\[
h+\delta=1+\rho,
\]
and hence
\[
\delta(1+\rho)+m_{\alpha,\beta}\geq0.
\]
These three conditions are equivalent to
\[
|\beta|\leq\frac12,
\quad
\delta\geq
\frac{(-m_{\alpha,\beta})_+}{1+\rho},
\]
which is precisely
\[
\sigma\leq\sigma_\uparrow(\alpha,\beta).
\]

Conversely, under these conditions,
\[
\delta(h+\delta)
\geq
\delta(1+\rho)
\geq
-m_{\alpha,\beta}
\]
for every $h\geq1+\sigma$. Since
\[
\Phi(x)\geq m_{\alpha,\beta},
\]
inequality~\eqref{eq:positive-bracket} follows. This proves the first
assertion.

The second assertion follows similarly; we record the details to
keep the signs explicit. The required inequality is equivalent to
\begin{equation}
\delta(h+\delta)+\Phi(x)\leq0,
\quad
h\geq1+\sigma,
\quad
0<x<\frac{\pi}{2}.
\label{eq:negative-bracket}
\end{equation}
Fixing $x$ and letting $h\to\infty$ shows that $\delta\leq0$, while
letting $x\uparrow\pi/2$ at fixed $h$ gives $B\leq0$. At
$h=1+\sigma$ one obtains
\[
\delta(1+\rho)+M_{\alpha,\beta}\leq0.
\]
Equivalently,
\[
|\beta|\geq\frac12,
\quad
-\delta\geq
\frac{(M_{\alpha,\beta})_+}{1+\rho},
\]
which is precisely
$\sigma\geq\sigma_\downarrow(\alpha,\beta)$.
Conversely, since $\delta\leq0$ and
$h+\delta\geq1+\rho$,
\[
\delta(h+\delta)
\leq
\delta(1+\rho)
\leq
-M_{\alpha,\beta}.
\]
Together with
\[
\Phi(x)\leq M_{\alpha,\beta},
\]
this proves~\eqref{eq:negative-bracket}.

It remains to examine the equality case. Suppose that, for some
fixed $t>0$, $Q_h(t)$ is constant on a nondegenerate interval of
admissible values of $h$. The function is real analytic for
$h>t/\pi$ and is therefore constant throughout
$(t/\pi,\infty)$. As $h\to\infty$,
\[
Q_h(t)
=
1+\frac{4A}{t^2}
+\frac{2\delta}{h}
+O(h^{-2}),
\]
so $\delta=0$. If $B\ne0$, then $Q_h(t)$ is unbounded as
$h\downarrow t/\pi$, a contradiction. Thus $B=0$. The two endpoint
limits are now
\[
\lim_{h\to\infty}Q_h(t)
=
1+\frac{4A}{t^2},
\quad
\lim_{h\downarrow t/\pi}Q_h(t)
=
1+\frac{\pi^2A}{t^2}.
\]
Since $\pi^2\ne4$, constancy implies $A=0$. Therefore
$A=B=\delta=0$. Conversely, these conditions give
$Q_h(t)\equiv1$, which proves the assertion.
\end{proof}

\begin{remark}
The intermediate values of $h$ in
Proposition~\ref{prop:sharp-potential} have no polynomial
interpretation. For fixed $\sigma$, the endpoint values
\[
h_n=n+\sigma,
\quad
h_{n+1}=n+1+\sigma
\]
correspond to degrees $n$ and $n+1$. The continuous interpolation
replaces a direct calculation of
$Q_{h_{n+1}}-Q_{h_n}$ by the exact
derivative~\eqref{eq:Q-derivative}.
Proposition~\ref{prop:sharp-potential} shows that the thresholds in
the two main theorems are, respectively, the largest and the smallest
values for which the corresponding monotonicity of the interpolating
potentials holds.
\end{remark}

\section{Increasing affinely scaled zeros}
\label{sec:increasing}

We now turn the potential comparison into a comparison of the zeros
of two consecutive scaled solutions.

\begin{proof}[Proof of Theorem~\ref{thm:increasing}]
For $n\geq1$, put
\[
h_n\coloneqq n+\sigma.
\]
Since $\sigma>-1$, one has $h_n>0$. In~\eqref{eq:scaled-equation}, the coefficient corresponding to
$Z_{n,h_n}$ is $Q_{h_n}$, because
\[
\frac{n+\rho}{h_n}
=
1+\frac{\rho-\sigma}{h_n}.
\]

Proposition~\ref{prop:sharp-potential} shows that, for each fixed
$t\in(0,h_n\pi)$, the function $h\mapsto Q_h(t)$ is nonincreasing
on $[h_n,h_{n+1}]$. Unless
\begin{equation}
A=B=0,
\quad
\sigma=\rho,
\label{eq:exceptional-conditions}
\end{equation}
it is not constant on this interval. If its values at the two
endpoints were equal, monotonicity would force it to be constant
throughout the interval. Hence
\begin{equation}
Q_{h_n}(t)>Q_{h_{n+1}}(t),
\quad
0<t<h_n\pi.
\label{eq:Q-decrease}
\end{equation}

Let
\[
a_k\coloneqq h_n\theta_{n,k},
\quad
b_k\coloneqq h_{n+1}\theta_{n+1,k}.
\]
If $b_k\geq h_n\pi$, then
\[
a_k<h_n\pi\leq b_k,
\]
and the desired inequality is immediate. If $b_k<h_n\pi$, apply
Lemma~\ref{lem:sturm} to $Z_{n,h_n}$ and
$Z_{n+1,h_{n+1}}$. The potential inequality
is~\eqref{eq:Q-decrease}, the Wronskian condition follows from
Lemma~\ref{lem:wronskian}, and positivity near $0$ follows
from~\eqref{eq:endpoint-expansion}. We obtain
\[
a_k<b_k,
\]
which is~\eqref{eq:increasing-intro}.

Finally,~\eqref{eq:exceptional-conditions} is equivalent to
\[
|\alpha|=|\beta|=\frac12,
\quad
\sigma=\rho.
\]
In that case, $Q_h\equiv1$. The endpoint behaviour
in~\eqref{eq:endpoint-expansion} identifies the corresponding solution
of $u''+u=0$. If $\alpha=-1/2$, then
\[
Z_{n,h_n}(0)=C_{n,h_n},
\quad
Z_{n,h_n}'(0)=0,
\]
and hence
\[
Z_{n,h_n}(t)=C_{n,h_n}\cos t.
\]
If $\alpha=1/2$, then
\[
Z_{n,h_n}(0)=0,
\quad
Z_{n,h_n}'(0)=C_{n,h_n},
\]
and hence
\[
Z_{n,h_n}(t)=C_{n,h_n}\sin t.
\]
In either case, the scaled zeros are independent of the degree, and
equality follows.
\end{proof}

\begin{lemma}\label{lem:explicit-threshold}
If $\alpha,\beta>-1$ and $|\beta|\leq1/2$, then~\eqref{eq:sigma-comparison} holds.
\end{lemma}

\begin{proof}
Put
\[
C\coloneqq A+3B
=
\frac{1-\alpha^2-3\beta^2}{4}.
\]
Since $B\geq0$, Lemma~\ref{lem:RT} gives
\begin{equation}
\Phi_{\alpha,\beta}(x)
=
C R(x)+B\bigl(T(x)-3R(x)\bigr)
\geq
C R(x).
\label{eq:profile-lower}
\end{equation}
If $C\geq0$, then $m_{\alpha,\beta}\geq0$, and consequently
\[
\sigma_\uparrow(\alpha,\beta)
=
\rho
=
\sigma_{\mathrm{alg}}(\alpha,\beta).
\]
If $C<0$, the inequality $R(x)<1$ shows
from~\eqref{eq:profile-lower} that $\Phi_{\alpha,\beta}(x)>C$. Hence
$m_{\alpha,\beta}\geq C$, and therefore
\[
\sigma_\uparrow(\alpha,\beta)
\geq
\rho+\frac{C}{1+\rho}
=
\sigma_{\mathrm{alg}}(\alpha,\beta).
\]

It remains to prove the second inequality in~\eqref{eq:sigma-comparison}. It is immediate when
$\alpha^2+3\beta^2\leq1$, because then
\[
\sigma_{\mathrm{alg}}=\rho>-\frac12.
\]
If $\alpha^2+3\beta^2>1$, direct expansion gives
\begin{align*}
&4(1+\rho)\left(\rho+\frac12\right)
-
\left(\alpha^2+3\beta^2-1\right)
\\[7pt]
&\quad
=
(2\beta+5)(\alpha+1)
+
2(2-\beta)\left(\beta+\frac12\right)
>0.
\end{align*}
Division by $4(1+\rho)>0$ completes the proof.
\end{proof}

\begin{corollary}\label{cor:explicit-increasing}
Let $\alpha,\beta>-1$ and $|\beta|\leq1/2$. Then~\eqref{eq:increasing-intro} holds whenever
\[
-1<\sigma\leq
\sigma_{\mathrm{alg}}(\alpha,\beta),
\]
with precisely the equality cases stated in
Theorem~\ref{thm:increasing}.
\end{corollary}

\begin{proof}
This follows from Lemma~\ref{lem:explicit-threshold} and
Theorem~\ref{thm:increasing}.
\end{proof}

We now deduce the first conjecture.

\begin{corollary}\label{cor:first-conjecture}
Conjecture~\ref{conj:first} holds throughout $\mathcal D_1$.
\end{corollary}

\begin{proof}
It is enough to show that
\begin{equation}
\sigma_{\mathrm{alg}}(\alpha,\beta)\geq0,
\quad
(\alpha,\beta)\in\mathcal D_1.
\label{eq:sigma-nonnegative}
\end{equation}
When $\alpha^2+3\beta^2\leq1$, this follows immediately from
$\sigma_{\mathrm{alg}}=\rho\geq0$.

Suppose that $\alpha^2+3\beta^2>1$. Then~\eqref{eq:sigma-nonnegative} is equivalent to
\begin{equation}
4\rho(1+\rho)
\geq
\alpha^2+3\beta^2-1.
\label{eq:S-algebra}
\end{equation}
The difference between the two sides is
\begin{equation}
2\bigl(
\alpha\beta-\beta^2+2\alpha+2\beta+2
\bigr).
\label{eq:S-difference}
\end{equation}
If $0\leq\beta\leq1/2$, the expression in parentheses is
increasing in $\alpha$, since its derivative with respect to
$\alpha$ is $\beta+2>0$. Its limiting value as
$\alpha\downarrow-1$ is
\[
\beta(1-\beta)\geq0.
\]
If $-1/2\leq\beta\leq0$, the expression is again increasing in
$\alpha$. The condition $\rho\geq0$ gives
$\alpha\geq-\beta-1$, and at this boundary value the difference in~\eqref{eq:S-difference} is
\[
-2\beta(2\beta+1)\geq0.
\]
Thus~\eqref{eq:S-algebra} holds in both cases.

Corollary~\ref{cor:explicit-increasing} may therefore be applied with
$\sigma=0$. Its only equality case with $\sigma=\rho=0$ is
\[
(\alpha,\beta)
=
\left(-\frac12,-\frac12\right),
\]
which is excluded from $\mathcal D_1$.
\end{proof}

\begin{remark}\label{rem:stronger-first}
For $\sigma\geq0$, the function
\[
\sigma\longmapsto
\frac{n+\sigma}{n+1+\sigma}
\]
is strictly increasing. Hence, whenever
$\sigma_\uparrow(\alpha,\beta)>0$, the choice
$\sigma=\sigma_\uparrow(\alpha,\beta)$ in
Theorem~\ref{thm:increasing} gives a strictly stronger estimate than
Conjecture~\ref{conj:first}.
\end{remark}

\section{Decreasing affinely scaled zeros}
\label{sec:decreasing}

The decreasing case follows the same scheme, with the ordering of the
potentials reversed. We include the argument to make the
strictness and the treatment of the singular endpoint explicit.

\begin{proof}[Proof of Theorem~\ref{thm:decreasing}]
Put $h_n\coloneqq n+\sigma$.
Since
\[
\sigma\geq\sigma_\downarrow(\alpha,\beta)\geq\rho,
\]
and $\rho>-1/2$, one has
\[
h_n\geq n+\rho>\frac12.
\]
Thus all the scales under consideration are positive. As in the proof
of Theorem~\ref{thm:increasing}, the potential corresponding to
$Z_{n,h_n}$ is $Q_{h_n}$, because
\[
\frac{n+\rho}{h_n}
=
1+\frac{\rho-\sigma}{h_n}.
\]

Fix $t\in(0,h_n\pi)$. For every
$h\in[h_n,h_{n+1}]$, one has $t<h\pi$, so that $Q_h(t)$ is defined
throughout this interval. Proposition~\ref{prop:sharp-potential}
shows that
\[
h\longmapsto Q_h(t)
\]
is nondecreasing on $[h_n,h_{n+1}]$. Unless
\begin{equation}
A=B=0,
\quad
\sigma=\rho,
\label{eq:decreasing-exceptional}
\end{equation}
the final assertion of that proposition shows that this function
cannot be constant on a nondegenerate interval. If its values at
$h_n$ and $h_{n+1}$ were equal, monotonicity would force it to be
constant throughout $[h_n,h_{n+1}]$. Consequently,
\begin{equation}
Q_{h_{n+1}}(t)>Q_{h_n}(t),
\quad
0<t<h_n\pi.
\label{eq:Q-increase}
\end{equation}

Set
\[
a_k\coloneqq h_{n+1}\theta_{n+1,k}^{(\alpha,\beta)},
\quad
b_k\coloneqq h_n\theta_{n,k}^{(\alpha,\beta)}.
\]
Since
\[
0<b_k<h_n\pi,
\]
Lemma~\ref{lem:sturm} may be applied on $(0,h_n\pi)$ with
\[
u_1=Z_{n+1,h_{n+1}},
\quad
q_1=Q_{h_{n+1}},
\quad
u_2=Z_{n,h_n},
\quad
q_2=Q_{h_n}.
\]
The strict potential inequality is~\eqref{eq:Q-increase}, the
Wronskian condition follows from Lemma~\ref{lem:wronskian},
and~\eqref{eq:endpoint-expansion} shows that both solutions are positive
in a sufficiently small right neighbourhood of the origin.
Therefore,
\[
a_k<b_k,
\]
or equivalently,
\[
h_{n+1}\theta_{n+1,k}^{(\alpha,\beta)}
<
h_n\theta_{n,k}^{(\alpha,\beta)}.
\]
This proves~\eqref{eq:decreasing-intro} outside~\eqref{eq:decreasing-exceptional}.

The conditions in~\eqref{eq:decreasing-exceptional} are equivalent to
\[
|\alpha|=|\beta|=\frac12,
\quad
\sigma=\rho.
\]
In this case $Q_h\equiv1$. As in the proof of
Theorem~\ref{thm:increasing}, the endpoint expansion determines the
corresponding solution of $u''+u=0$: it is a positive multiple of
$\cos t$ when $\alpha=-1/2$ and a positive multiple of $\sin t$ when
$\alpha=1/2$. Its positive zeros are therefore independent of the
degree, and equality holds for every admissible $n$ and $k$.
Conversely, the strict comparison above excludes equality in every
other case.
\end{proof}

\begin{corollary}\label{cor:second-conjecture}
For every $(\alpha,\beta)\in\mathcal D_\downarrow$,
$n\geq1$, and $1\leq k\leq n$,
\[
\gamma_{n+1}\theta_{n+1,k}^{(\alpha,\beta)}
\leq
\gamma_n\theta_{n,k}^{(\alpha,\beta)}.
\]
Equality occurs for some $n$ and $k$ if and only if
\[
|\alpha|=|\beta|=\frac12;
\]
at these four parameter values it occurs for every $n$ and $k$.
In particular, Conjecture~\ref{conj:second} holds.
\end{corollary}

\begin{proof}
For $(\alpha,\beta)\in\mathcal D_\downarrow$, put
\[
C\coloneqq A+3B
=
\frac{1-\alpha^2-3\beta^2}{4}.
\]
The defining conditions of $\mathcal D_\downarrow$ give
\[
B\leq0,
\quad
C\leq0.
\]
Proposition~\ref{prop:spectral-sign} therefore gives
$\Phi_{\alpha,\beta}(x)\leq0$ throughout $(0,\pi/2)$.
It follows that $M_{\alpha,\beta}\leq0$ and therefore
\[
\sigma_\downarrow(\alpha,\beta)=\rho.
\]
The conclusion, including the equality statement, now follows from
Theorem~\ref{thm:decreasing} with $\sigma=\rho$.

Finally,
\[
\mathcal D_2\subsetneq\mathcal D_\downarrow,
\]
and the four equality points are excluded from
$\mathcal D_2$ by its definition. The inequality is therefore strict
throughout the domain of Conjecture~\ref{conj:second}.
\end{proof}

\begin{corollary}\label{cor:reverse}
For every $(\alpha,\beta)\in\mathcal D_\uparrow$,
$n\geq1$, and $1\leq k\leq n$,
\[
\gamma_n\theta_{n,k}^{(\alpha,\beta)}
\leq
\gamma_{n+1}\theta_{n+1,k}^{(\alpha,\beta)}.
\]
Equality occurs for some $n$ and $k$ if and only if
\[
|\alpha|=|\beta|=\frac12;
\]
at these four parameter values it occurs for every $n$ and $k$.
\end{corollary}

\begin{proof}
For $(\alpha,\beta)\in\mathcal D_\uparrow$, one has
\[
B\geq0, \quad A+3B
=
\frac{1-\alpha^2-3\beta^2}{4}
\geq0.
\]
Proposition~\ref{prop:spectral-sign} therefore gives
$\Phi_{\alpha,\beta}(x)\geq0$ throughout $(0,\pi/2)$.
Consequently, $m_{\alpha,\beta}\geq0$ and
$\sigma_\uparrow(\alpha,\beta)=\rho$.
The result follows from Theorem~\ref{thm:increasing} with
$\sigma=\rho$, including the asserted characterisation of equality.
\end{proof}

\section{Sharpness of the comparison regions}
\label{sec:spectral-sharpness}

Let $j_{\alpha,k}$ denote the $k$th positive zero of the Bessel
function $J_\alpha$. To make the analytic object used below explicit,
write
\[
\mathcal J_\alpha(z)
\coloneqq
\left(\frac z2\right)^{-\alpha}J_\alpha(z)
=
\sum_{r=0}^{\infty}
\frac{(-1)^r(z^2/4)^r}{r!\,\Gamma(r+\alpha+1)}.
\]
Thus $\mathcal J_\alpha$ is an entire even function, including when
$\alpha$ is not an integer, and its positive zeros are precisely the
numbers $j_{\alpha,k}$. The Mehler--Heine formula states that
\[
n^{-\alpha}
P_n^{(\alpha,\beta)}\!\left(\cos\frac zn\right)
\longrightarrow
\mathcal J_\alpha(z)
\]
uniformly on compact subsets of the complex plane; see
Szegő~\cite[Theorem~8.1.1]{Szego}. Hurwitz's theorem therefore gives,
for every fixed $k$,
\begin{equation}
\lim_{n\to\infty}
n\theta_{n,k}^{(\alpha,\beta)}
=
j_{\alpha,k};
\label{eq:zero-limit}
\end{equation}
Consequently, for every fixed real $\sigma$,
\begin{equation}
\lim_{n\to\infty}
(n+\sigma)\theta_{n,k}^{(\alpha,\beta)}
=
j_{\alpha,k}.
\label{eq:affine-zero-limit}
\end{equation}

We first record the Bessel-zero spacing needed below.

\begin{lemma}\label{lem:bessel-spacing}
Let $\nu>-1$. Then
\[
j_{\nu,2}-j_{\nu,1}
\begin{cases}
<\pi,&|\nu|<1/2,\\[7pt]
=\pi,&|\nu|=1/2,\\[7pt]
>\pi,&|\nu|>1/2.
\end{cases}
\]
\end{lemma}

\begin{proof}
Put
\[
u(x)\coloneqq\sqrt{x}J_\nu(x).
\]
The Bessel equation gives
\[
u''(x)
+
\left(
1+\frac{1/4-\nu^2}{x^2}
\right)u(x)
=0,
\quad x>0.
\]
Let
\[
a\coloneqq j_{\nu,1},
\quad
d\coloneqq j_{\nu,2}-j_{\nu,1}.
\]
After changing the sign of $u$ if necessary, the function
\[
U(t)\coloneqq u(a+t)
\]
is positive in a right neighbourhood of $0$, and its first positive
zero is $d$. Compare it with $V(t)=\sin t$. Both functions vanish at
$t=0$, so the boundary Wronskian condition in Lemma~\ref{lem:sturm}
is immediate. Fix $L>\max\{d,\pi\}$.

If $|\nu|<1/2$, then
\[
1+\frac{1/4-\nu^2}{(a+t)^2}>1,
\quad t>0.
\]
Lemma~\ref{lem:sturm}, with $U$ as the solution corresponding to the
larger coefficient, gives $d<\pi$. If $|\nu|>1/2$, the coefficient
inequality is reversed, and the same lemma, with the roles of $U$ and
$V$ interchanged, gives $\pi<d$. Finally, if $|\nu|=1/2$, the
equation for $u$ is $u''+u=0$, so consecutive positive zeros are
separated by $\pi$.
\end{proof}

\begin{proof}[Proof of Corollary~\ref{cor:increasing-affine-classification}]
Suppose first that $|\beta|\leq1/2$. By~\eqref{eq:sigma-comparison},
$\sigma_\uparrow(\alpha,\beta)>-1/2$, so the interval
\[
-1<\sigma\leq\sigma_\uparrow(\alpha,\beta)
\]
is nonempty. Theorem~\ref{thm:increasing} therefore gives the required
comparison for every $\sigma$ in this interval.

Conversely, suppose that $|\beta|>1/2$ and fix $\sigma>-1$. Put
\[
D_n
\coloneqq
(n+1+\sigma)\theta_{n+1,n}^{(\alpha,\beta)}
-
(n+\sigma)\theta_{n,n}^{(\alpha,\beta)}.
\]
The Jacobi symmetry
\[
P_m^{(\alpha,\beta)}(-x)
=
(-1)^mP_m^{(\beta,\alpha)}(x)
\]
gives
\[
D_n
=
\pi
+
(n+\sigma)\theta_{n,1}^{(\beta,\alpha)}
-
(n+1+\sigma)\theta_{n+1,2}^{(\beta,\alpha)}.
\]
By~\eqref{eq:affine-zero-limit},
\[
\lim_{n\to\infty}D_n
=
\pi+j_{\beta,1}-j_{\beta,2}.
\]
Lemma~\ref{lem:bessel-spacing} shows that this limit is negative.
Hence $D_n<0$ for all sufficiently large $n$, so the increasing
comparison fails for the zero index $k=n$. Since $\sigma>-1$ was
arbitrary, no degree-independent affine shift can give the increasing
ordering for all $n$ and $k$.
\end{proof}

\begin{proof}[Proof of Proposition~\ref{prop:spectral-failure}]
Put
\[
C\coloneqq A+3B,
\quad
\Delta_{n,k}
\coloneqq
\gamma_{n+1}\theta_{n+1,k}^{(\alpha,\beta)}
-
\gamma_n\theta_{n,k}^{(\alpha,\beta)}.
\]
We first determine the sign of $\Delta_{n,1}$. In
Lemma~\ref{lem:potential-derivative}, take $\sigma=\rho$, so that
$\delta=0$. Then
\begin{equation}
\frac{\partial Q_h(t)}{\partial h}
=
-\frac{2}{h^3}
\Phi_{\alpha,\beta}\left(\frac{t}{2h}\right).
\label{eq:spectral-Q-derivative}
\end{equation}
Moreover, by~\eqref{eq:RT-zero},
\[
\lim_{x\downarrow0}\Phi_{\alpha,\beta}(x)=\frac{C}{3}.
\]

Suppose first that $C>0$. Choose $\varepsilon>0$ such that
\[
\Phi_{\alpha,\beta}(x)>0,
\quad 0<x<\varepsilon,
\]
and fix $T>j_{\alpha,1}$. By~\eqref{eq:affine-zero-limit}, for all
sufficiently large $n$ the first positive zeros of
$Z_{n,\gamma_n}$ and $Z_{n+1,\gamma_{n+1}}$ both lie in $(0,T)$.
Increasing $n$ if necessary, we may also assume that
\[
T<\gamma_n\pi,
\quad
\frac{T}{2\gamma_n}<\varepsilon.
\]
If $\gamma_n\leq h\leq\gamma_{n+1}$ and $0<t<T$, then
\[
0<\frac{t}{2h}<\varepsilon.
\]
Hence~\eqref{eq:spectral-Q-derivative} gives
\[
Q_{\gamma_n}(t)>Q_{\gamma_{n+1}}(t),
\quad 0<t<T.
\]
The Wronskian condition follows from Lemma~\ref{lem:wronskian}, and
positivity near $0$ follows from~\eqref{eq:endpoint-expansion}.
Lemma~\ref{lem:sturm} therefore gives
\[
\gamma_n\theta_{n,1}^{(\alpha,\beta)}
<
\gamma_{n+1}\theta_{n+1,1}^{(\alpha,\beta)},
\]
so $\Delta_{n,1}>0$. If $C<0$, the same argument with the potential
inequality reversed gives $\Delta_{n,1}<0$.

We next consider $\Delta_{n,n}$. The Jacobi symmetry
\[
P_m^{(\alpha,\beta)}(-x)
=
(-1)^mP_m^{(\beta,\alpha)}(x)
\]
gives
\[
\theta_{m,k}^{(\alpha,\beta)}
=
\pi-\theta_{m,m+1-k}^{(\beta,\alpha)}.
\]
Consequently,
\[
\Delta_{n,n}
=
\pi
+
\gamma_n\theta_{n,1}^{(\beta,\alpha)}
-
\gamma_{n+1}\theta_{n+1,2}^{(\beta,\alpha)}.
\]
Since $\rho$ is unchanged when $\alpha$ and $\beta$ are
interchanged,~\eqref{eq:affine-zero-limit} yields
\[
\lim_{n\to\infty}\Delta_{n,n}
=
\pi+j_{\beta,1}-j_{\beta,2}.
\]
Lemma~\ref{lem:bessel-spacing} now shows that
$\Delta_{n,n}$ has the same sign as $B$ for all sufficiently large
$n$. Since $BC<0$ on $\mathcal U$, the
integer $N$ may be chosen so that
\[
\Delta_{n,1}\Delta_{n,n}<0,
\quad n\geq N.
\]
This proves the proposition.
\end{proof}

\begin{proof}[Proof of Corollary~\ref{cor:spectral-classification}]
The sufficiency of the two conditions follows from
Corollaries~\ref{cor:reverse} and~\ref{cor:second-conjecture},
respectively. Suppose that the first inequality in the statement holds
for every $n$ and $k$, but
$(\alpha,\beta)\notin\mathcal D_\uparrow$. If
$(\alpha,\beta)\in\mathcal U$, this contradicts
Proposition~\ref{prop:spectral-failure}. Otherwise
$(\alpha,\beta)\in\mathcal D_\downarrow\setminus
\mathcal D_\uparrow$, and Corollary~\ref{cor:second-conjecture}
gives the strict opposite inequality. Thus
$(\alpha,\beta)\in\mathcal D_\uparrow$.

The necessity of the condition for the second inequality is proved in
the same way, using Corollary~\ref{cor:reverse} on
$\mathcal D_\uparrow\setminus\mathcal D_\downarrow$. This proves both
equivalences.
\end{proof}

\section{Bounds in terms of Bessel zeros}
\label{sec:bessel-bounds}

The affine factor in~\eqref{eq:affine-zero-limit} does not alter the
limiting Bessel zero. The comparison theorems convert this asymptotic
relation into one-sided bounds valid at every finite degree. Such
one-sided convergence was one of the motivations for Gautschi's
scaled-zero conjectures; see~\cite[Section~1]{Gautschi2009b}.

\begin{corollary}\label{cor:bessel-bounds}
Let $n\geq1$ and $1\leq k\leq n$.
\begin{enumerate}
\item If $\alpha,\beta>-1$, $|\beta|\leq1/2$, and
$-1<\sigma\leq\sigma_\uparrow(\alpha,\beta)$,
then
\[
\theta_{n,k}^{(\alpha,\beta)}
\leq
\frac{j_{\alpha,k}}{n+\sigma}.
\]
\item If $\alpha,\beta>-1$, $|\beta|\geq1/2$, and
$\sigma\geq\sigma_\downarrow(\alpha,\beta)$,
then
\[
\theta_{n,k}^{(\alpha,\beta)}
\geq
\frac{j_{\alpha,k}}{n+\sigma}.
\]
\end{enumerate}
In either part the inequality is strict unless
\[
|\alpha|=|\beta|=\frac12, \quad
\sigma=\frac{\alpha+\beta+1}{2},
\]
in which case equality holds for every admissible $n$ and $k$.
\end{corollary}

\begin{proof}
Fix $n$ and $k$, and, for $m\geq n$, put
\[
a_m\coloneqq
(m+\sigma)\theta_{m,k}^{(\alpha,\beta)}.
\]
Under the assumptions of the first part,
Theorem~\ref{thm:increasing} gives
\[
a_m\leq a_{m+1},
\quad
m\geq n.
\]
By~\eqref{eq:zero-limit},
\[
\lim_{m\to\infty}a_m=j_{\alpha,k},
\]
and hence
\[
a_n\leq j_{\alpha,k}.
\]
This is the first assertion.

Under the assumptions of the second part,
Theorem~\ref{thm:decreasing} gives instead
\[
a_{m+1}\leq a_m,
\quad
m\geq n.
\]
The same limiting relation now yields
\[
a_n\geq j_{\alpha,k},
\]
which proves the second assertion.

Outside the exceptional case, the corresponding comparison theorem
makes every consecutive inequality strict. A strictly increasing
sequence converging to $j_{\alpha,k}$ lies strictly below its limit,
whereas a strictly decreasing one lies strictly above it. In the
exceptional case the sequence is constant, and its common value must
equal its limit $j_{\alpha,k}$. This proves the equality statement.
\end{proof}

For the upper bound in Corollary~\ref{cor:bessel-bounds}, the
right-hand side decreases as $\sigma$ increases. The strongest choice
provided by Theorem~\ref{thm:increasing} is therefore
\[
\sigma=\sigma_\uparrow(\alpha,\beta).
\]
For the lower bound, the strongest choice is the smallest admissible
value,
\[
\sigma=\sigma_\downarrow(\alpha,\beta).
\]
When admissible in the relevant comparison theorem, the choices
$\sigma=0$ and $\sigma=\rho$ recover the distinguished denominators
$n$ and
\[
\gamma_n=n+\rho,
\]
respectively.

A particularly transparent two-sided estimate arises at the common
parameter-boundary point
\[
(\alpha,\beta)=\left(0,\frac12\right).
\]
Here
\[
\rho=\frac34,
\quad
A=\frac1{16},
\quad
B=0,
\]
so that
\[
\Phi_{0,1/2}(x)=\frac1{16}R(x).
\]
Since $R(x)>0$ on $(0,\pi/2)$ and
\[
\sup_{0<x<\pi/2}R(x)=1,
\]
the two affine thresholds are
\[
\sigma_\uparrow\left(0,\frac12\right)=\frac34, \quad
\sigma_\downarrow\left(0,\frac12\right)
=
\frac34+
\frac{1/16}{1+3/4}
=
\frac{11}{14}.
\]
Corollary~\ref{cor:bessel-bounds} therefore gives
\begin{equation}
L_{n,k}
\coloneqq
\frac{j_{0,k}}{n+11/14}
<
\theta_{n,k}^{(0,1/2)}
<
\frac{j_{0,k}}{n+3/4}
\eqqcolon
U_{n,k},
\label{eq:two-sided-example}
\end{equation}
for every $n\geq1$ and $1\leq k\leq n$. The length of this
enclosure is given exactly by
\begin{equation}
U_{n,k}-L_{n,k}
=
\frac{j_{0,k}}
{28(n+3/4)(n+11/14)}.
\label{eq:two-sided-width}
\end{equation}
Moreover, since $\theta_{n,k}^{(0,1/2)}>L_{n,k}$, one obtains the
uniform relative estimate
\begin{equation}
0<
\frac{U_{n,k}-L_{n,k}}
     {\theta_{n,k}^{(0,1/2)}}
<
\frac{1}{28n+21}.
\label{eq:relative-width}
\end{equation}
Thus the relative width is $O(n^{-1})$, uniformly with respect to the
zero index $k$. Table~\ref{tab:two-sided-bounds} illustrates this
enclosure for zeros situated near the beginning, middle, and end of
the zero set.

\begin{table}[htbp]
\centering
\caption{Representative two-sided enclosures from~\eqref{eq:two-sided-example}. The entries in the
$\theta_{n,k}^{(0,1/2)}$ column are numerical approximations to the
angular zeros; the relative width is
$(U_{n,k}-L_{n,k})/\theta_{n,k}^{(0,1/2)}$.}
\label{tab:two-sided-bounds}

\renewcommand{\arraystretch}{1.10}
\setlength{\tabcolsep}{7pt}

\begin{tabular}{rrcccc}
\toprule
\rowcolor{gray!15}
$n$ & $k$
& $L_{n,k}$
& $\theta_{n,k}^{(0,1/2)}$
& $U_{n,k}$
& Relative width \\
\midrule

2  & 1  & 0.863270713 & 0.873269898 & 0.874482021 & $1.284\%$ \\
2  & 2  & 1.981566501 & 2.004353607 & 2.007301131 & $1.284\%$ \\

\addlinespace[3pt]
\rowcolor{gray!6}
5  & 1  & 0.415648862 & 0.418098575 & 0.418230532 & $0.617\%$ \\
\rowcolor{gray!6}
5  & 3  & 1.495706059 & 1.504503879 & 1.504996159 & $0.617\%$ \\
\rowcolor{gray!6}
5  & 5  & 2.580652443 & 2.595755459 & 2.596681341 & $0.618\%$ \\

\addlinespace[3pt]
10 & 1  & 0.222963959 & 0.223684530 & 0.223704703 & $0.331\%$ \\
10 & 5  & 1.384323496 & 1.388793222 & 1.388922578 & $0.331\%$ \\
10 & 10 & 2.840294639 & 2.849430950 & 2.849730834 & $0.331\%$ \\

\bottomrule
\end{tabular}
\end{table}

The angular zeros in the central column were computed from the zeros
of $P_n^{(0,1/2)}$ using 80-digit working precision. The computation
was independently checked by forming the Jacobi matrix associated
with the three-term recurrence and diagonalising it; the resulting
eigenvalues were converted to angles by $\arccos$. All numerical
entries were then rounded to the number of decimal places displayed.
The near-constant relative widths illustrate the uniformity in $k$
already quantified by~\eqref{eq:relative-width}, including for
indices that vary with the degree.

\section{Concluding remarks}

The affine scale $h_n=n+\sigma$ places Gautschi's two conjectures
within a single differential comparison. On their respective
parameter domains, setting $\sigma=0$ gives the first conjecture and
makes the rescaled potential decrease with $h$, whereas setting
$\sigma=\rho$ gives the second and makes the rescaled potential
increase with $h$. The
quantities
$\sigma_\uparrow(\alpha,\beta)$ and
$\sigma_\downarrow(\alpha,\beta)$ are the exact thresholds for these
two pointwise monotonicity properties. They are therefore sharp for
the continuous potential comparison developed here. For general
affine scales no claim is made that they are necessary for the
corresponding inequalities between individual zeros. Nevertheless,
Corollary~\ref{cor:increasing-affine-classification} shows that
$|\beta|\leq1/2$ is necessary and sufficient for the existence of some
fixed affine shift producing an increasing ordering for every degree
and zero index. For the spectral scale $\sigma=\rho$,
Corollary~\ref{cor:spectral-classification} shows that the parameter
regions $\mathcal D_\uparrow$ and $\mathcal D_\downarrow$ are exact.

For the spectral scale, the identity
\[
\Phi_{\alpha,\beta}(x)
=
(A+3B)R(x)+B\bigl(T(x)-3R(x)\bigr)
\]
and the inequality $T(x)>3R(x)$ make the geometry of the parameter
regions transparent. The condition
\[
A+3B=0
\]
is equivalent to the ellipse
\[
\alpha^2+3\beta^2=1,
\]
while $B=0$ corresponds to $|\beta|=1/2$. These two boundaries
delimit the two exact spectral-ordering regions. Their intersections
are the four points
\[
|\alpha|=|\beta|=\frac12,
\]
which are precisely the cases in which the rescaled potential is
constant and equality holds. Together with the Mehler--Heine limit,
the comparison also gives the finite-degree Bessel-zero bounds in
Corollary~\ref{cor:bessel-bounds}.

The endpoint analysis is an essential part of the argument. Each
scaled Jacobi-polynomial solution $Z_{n,h}$ has leading Frobenius
behaviour
\[
C_{n,h}t^{\alpha+1/2}
\]
at the origin. When $-1<\alpha<-1/2$, the transformed solutions are
unbounded there, so no finite endpoint value can be assigned to them.
Nevertheless, the leading terms cancel in the boundary Wronskian,
which satisfies
\[
Z_{m,h}'Z_{\ell,g}-Z_{m,h}Z_{\ell,g}'
=
O\bigl(t^{2\alpha+2}\bigr).
\]
Since $\alpha>-1$, this expression tends to zero. This cancellation
is exactly what the singular Sturm comparison requires and allows
this endpoint comparison step to remain valid throughout the full
Jacobi range $\alpha,\beta>-1$ without any implicit boundedness
assumption at the endpoint.

In the two remaining sign configurations,
$B$ and $A+3B$ have strictly opposite signs. Consequently,
$\Phi_{\alpha,\beta}$ changes sign on $(0,\pi/2)$ and the pointwise
potential comparison has no uniform direction.
Proposition~\ref{prop:spectral-failure} shows that this reflects a
genuine loss of spectral ordering: for every parameter point in
$\mathcal U$, the first and last comparable scaled zeros satisfy
opposite inequalities for all sufficiently large degrees. Thus
Corollary~\ref{cor:spectral-classification} gives a complete parameter
classification for the spectral scale.

\section*{Acknowledgements}
K.~Castillo acknowledges financial support from the Centre for
Mathematics of the University of Coimbra (CMUC), funded by the
Portuguese Foundation for Science and Technology (FCT), under the
projects UID/00324/2025 and UID/PRR/00324/2025. K.~Castillo also
acknowledges financial support from the FCT under the grant
2022.00143.CEECIND/CP1714/CT0002.

\end{document}